\documentclass[12pt]{amsart}
\usepackage[colorlinks=true,urlcolor=blue, citecolor=red,linkcolor=blue,linktocpage,pdfpagelabels, bookmarksnumbered,bookmarksopen]{hyperref}
\usepackage[hyperpageref]{backref}
\usepackage{amsthm} 
\usepackage{latexsym,amsmath,amssymb}
\usepackage{accents}
\usepackage[colorinlistoftodos,prependcaption,textsize=tiny]{todonotes}
\usepackage{a4wide}
\usepackage{soul}
\usepackage{mathtools} 
\usepackage{xparse} 
  
\title{On $C^{1,\alpha}$-regularity for critical points of a geometric obstacle-type problem}

\author{Sujin Khomrutai}
\address[Sujin Khomrutai]{Department of Mathematics and Computer Science, Faculty of Science, Chulalongkorn University, Bangkok 10330, Thailand}
\email{sujin.k@chula.ac.th}

\author{Armin Schikorra}
\address[Armin Schikorra]{Department of Mathematics,
University of Pittsburgh,
301 Thackeray Hall,
Pittsburgh, PA 15260, USA}
\email{armin@pitt.edu}

plus 6pt minus 12pt
\def\eps{\varepsilon}

\def\B{{\mathbb{B}}}

\def\N{{\mathbb N}}

\def\S{{\mathbb S}}

\newtheorem{theorem}{Theorem}
\newtheorem{lemma}[theorem]{Lemma}
\newtheorem{corollary}[theorem]{Corollary}
\newtheorem{proposition}[theorem]{Proposition}

\newtheorem{remark}[theorem]{Remark}

\def\dist{{\rm dist\,}}

\def\curl{{\rm curl\,}}

\def\supp{{\rm supp\,}}

\newcommand{\R}{\mathbb{R}}

\newcommand{\brac}[1]{\left (#1 \right )}

\newcommand{\barint}{
\rule[.036in]{.12in}{.009in}\kern-.16in \displaystyle\int }

\newcommand{\barcal}{\mbox{$ \rule[.036in]{.11in}{.007in}\kern-.128in\int $}}

\def\mvint_#1{\mathchoice
          {\mathop{\vrule width 6pt height 3 pt depth -2.5pt
                  \kern -8pt \intop}\nolimits_{\kern -3pt #1}}%
          {\mathop{\vrule width 5pt height 3 pt depth -2.6pt
                  \kern -6pt \intop}\nolimits_{#1}}%
          {\mathop{\vrule width 5pt height 3 pt depth -2.6pt
                  \kern -6pt \intop}\nolimits_{#1}}%
          {\mathop{\vrule width 5pt height 3 pt depth -2.6pt
                  \kern -6pt \intop}\nolimits_{#1}}}

\numberwithin{theorem}{section} \numberwithin{equation}{section}

\newcommand{\lap}{\Delta }
\newcommand{\aleq}{\precsim}

\renewcommand{\div}{\operatorname{div}}

\let\latexchi\chi
\makeatletter
\renewcommand\chi{\@ifnextchar_\sub@chi\latexchi}
\newcommand{\sub@chi}[2]{
  \@ifnextchar^{\subsup@chi{#2}}{\latexchi^{}_{#2}}%
}
\newcommand{\subsup@chi}[3]{
  \latexchi_{#1}^{#3}%
}
\makeatother

\begin{document}

\begin{abstract}
We consider critical points of the geometric obstacle problem on vectorial maps $u: \B^2 \subset \R^2 \to \R^N$
\[
 \int_{\B^2} |\nabla u|^2 \quad \mbox{subject to $u \in \R^N \backslash \B^N(0)$}.
\]
Our main result is $C^{1,\alpha}$-regularity for any $\alpha < 1$. 

Technically, we split the map $u=\lambda v$, where $v: \B^2 \to \S^{N-1}$ is the vectorial component and $\lambda = |u|$ the scalar component measuring the distance to the origin. While $v$ satisfies a weighted harmonic map equation with weight $\lambda^2$, $\lambda$ solves the obstacle problem for
\[
 \int_{\B^2} |\nabla \lambda|^2+\lambda^2 |\nabla v|^2, \quad \mbox{subject to $\lambda \geq 1$}.
\]
where $|\nabla v|^2 \in L^1(\B^2)$. We then play ping-pong between the increases in the regularity of $\lambda$ and $v$ to obtain finally the $C^{1,\alpha}$-result.
\end{abstract}

\sloppy

\maketitle
\tableofcontents
\sloppy
\section{Introduction}

Denote by
\[
D(u) :=  \int_{\B^2} |\nabla u|^2.
\]
the Dirichlet energy for maps defined on the two-dimensional disk $\B^2 \subset \R^2$. 

The classical obstacle problem for a given obstacle function $\omega: \B^2 \to \R$ analyzes the minimizer
\[
\inf_{ f \geq \omega} \int_{B^2} |\nabla f|^2 
\]
One can reformulate the obstacle problems for graphs $u = (x,f(x))$  as analyzing the minimizer of the problem
\[
\inf_{ X_\Omega} \int_{B^2} |\nabla u|^2,
\]
where 
\[
\Omega = \left \{ (x,t) \in \B^2\times \R: t < \varphi(x) \right \}
\]
and the infimum is taken over the set of maps not touching $\Omega$.
\begin{equation}\label{eq:obstacleclass}
X_\Omega := \left \{u \in H^1(\B^2,\R^3):\  u \not \in  \Omega\right \}
\end{equation}
It is a natural to consider this situation for sets $\Omega$ whose boundary is smooth and compact, but which may not be a graph. In this case, $u$ can be thought of as a soap film in three-dimensional space which lives outside of a solid ball. Where the soap film intersects with the solid ball, a free boundary appears. 

Geometric obstacle problems have been considered, e.g. \cite{HN82} but this is quite different from our case. Much closer to our situation, considering minimizers, is the setup as in \cite{H73,DF86}. Since the obstacle problem is not convex anymore, it is natural to consider not only minimizers but also critical points, which we shall do in this work.

A first observation is that the geometric setting immediately leads to regularity issues: while in the classical obstacle theory, basic $C^{1,\alpha}$-regularity is quite easy to obtain, already for the simplest case of round obstacles $\Omega = \B^{N-1}$, any harmonic function into $\partial \B^{N-1} = \S^{N-1}$ is necessarily a critical point of the obstacle problem. Indeed we have,
\begin{proposition}
Let $\bar{v}$ be a minimizing harmonic map from $\B^n \to \S^{N-1}$ with respect to its own boundary values, then $u:= \bar{v}$ minimizes the Dirichlet energy in the class $X_{\B^{N-1}}$ with respect to its own boundary values.

If $\bar{v}$ is a critical (possibly non-minimizing) harmonic map from $\B^2 \to \S^{N-1}$, then $\bar{v}$ is a critical map for the Dirichlet energy with respect tot he class $X_{\B^{N-1}}$.
\end{proposition}
\begin{proof}
We split $u = \lambda v$, where $\lambda = |u| \geq 1$ and $v = \frac{u}{|u|} \in H^1(\B^2,\S^{2})$. Since $v \cdot \nabla v \equiv 0$, we have
\[
|\nabla u|^2 = |\nabla \lambda\, v + \lambda \nabla v|^2 = |\nabla \lambda|^2 + |\lambda|^2 |\nabla v|^2
\]
In particular,
\[
\int_{\B^n} |\nabla u|^2 \geq \int |\nabla v|^2
\]
with equality if and only if $\lambda \equiv 1$.
The conclusion now follows.
\end{proof}
In particular, for $n \geq 3$ there is no hope of obtaining even mere continuity at the free boundary for the solutions of the obstacle problem: harmonic maps may only be smooth for $n \geq 3$ on a large set (not everywhere), see \cite{SU84,DF86}, and if we consider critical harmonic maps may be everywhere discontinuous, see \cite{R95}. 

This is why, for now, we shall restrict our attention to $n=2$. The main result of this work is the basic regularity theory for spherical obstacles.

\begin{theorem}\label{th:main}
Let $\Omega = \B^{N-1} \subset \R^N$ be the solid unit ball. Denote the obstacle class $X_\Omega$ as in \eqref{eq:obstacleclass}. Then any map of which is critical in $X_\Omega$ with respect to $D(\cdot)$ is $C^{1,\alpha}$-smooth, for any $\alpha < 1$.
\end{theorem}

In future works we plan to analyze the free boundary, where $u$ intersects with $\partial \Omega$, as well as more general obstacles.

Let us also state that as a by-product of our arguments we obtain the following regularity result for harmonic maps into the (non-compact) manifold of conformal transformations.
\begin{theorem}\label{th:mainSO}
Denote the group of conformal transformations with conformal factor bounded from below by $\lambda_0$ as \[CO_{\lambda_0}(N) = \left \{\lambda\, Q  \in \R^{N \times N}: Q \in SO(N), \lambda > \lambda_0 \right \}\] 

Then for $\lambda_0 > 0$, any map $P \in H^1(\B^2,CO_{\lambda_0}(N))$ which is a critical point of the Dirichlet energy $D(\cdot)$ in the class of maps into $CO_{\lambda_0}(N)$ belongs to $C^{1,\alpha}$ for any $\alpha < 1$.
\end{theorem}
The proof is almost verbatim to the one of Theorem~\ref{th:main}, we point out the differences in Section~\ref{s:CO}.

\subsection{A reformulation of Theorem~\ref{th:main}}
In order to prove Theorem~\ref{th:main} we represent any point $u$ in $\R^N \backslash \B^N$ uniquely as
\[
u = \lambda v,
\]
where $v = \frac{u}{|u|} \in \S^{N-1}$ and $\lambda = |u| >0$.

If $u \in H^1(\B^2,\R^n \backslash \B_1^N)$ then $\lambda \in H^1(\B^2)$ is a scalar function and $v \in H^1(\B^2,\S^{N-1})$. In particular we have $\langle v, \nabla v \rangle = 0$, which leads to
\[
|\nabla u|^2 = |\nabla \lambda \, v + \lambda \nabla v|^2 = |\nabla \lambda|^2 + \lambda^2 |\nabla v|^2. 
\]
Consequently, Theorem~\ref{th:main} can be reformulated as
\begin{theorem}\label{th:main2}
Let $(\lambda,v) \in H^1(\B^2) \times H^1(\B^2,\S^{N-1})$ be a critical map with respect to the energy
\[
E(\lambda,v) := \int |\nabla \lambda|^2 + \int \lambda^2 |\nabla v|^2
\]
and subject to $\lambda \geq \lambda_0$. That is, 
\begin{itemize}
\item assume that
\[
\frac{d}{d\eps}\Big |_{\eps = 0} E(\lambda + \eps \varphi,v)  \geq 0
\]
holds whenever $\varphi \in H^1_0(\B^2)$, and $\lambda + \eps \varphi \geq \lambda$ almost everywhere in $\B^2$ and $(\lambda + \eps \varphi) v \in H^1(\B^2)$ for small $\eps$.
\item and 
\[
\frac{d}{d\eps}\Big |_{\eps = 0} E \left (\lambda ,\frac{v+\eps \psi}{|v + \eps \psi|} \right )  = 0
\]
holds for any $\psi \in C_c^\infty(\B^2,\R^N)$
\end{itemize}
Then $u = \lambda v \in C^{1,\alpha}$ for some $\alpha > 0$.
\end{theorem}
\begin{remark}
\begin{itemize}
\item By an easy adaptation of the proof one can show that $\lambda \geq 1$ can be replaced by $\lambda \geq \lambda_0$ where $\lambda_0 \in C^\infty(\overline{\B^2},(0,\infty))$ with $\inf_{\B^2} \lambda_0 > 0$. Observe that e.g. for starshaped obstacles the approach is much more complicated: Then one would need to assume $\lambda \geq \lambda_0(v)$, i.e. have to consider an obstacle depending on $v$, which heavily complicates the variation in $v$.

\item Moreover, observe that $E$ as above is convex in $\lambda$, but not in $v$. That is, the only critical points in terms of $\lambda$ are minimizers, but again not necessarily so $v$.
\end{itemize}
\end{remark}

\subsection{Outline of the proof}
The proof of Theorem~\ref{th:main2} is split into several parts, since we have to jump between improvements in regularity of $\lambda$ and $v$. 
First we prove in Section~\ref{s:lbounded} local boundedness of $\lambda$, see Proposition~\ref{pr:lambdabd}. Then we compute the Euler-Lagrange equations for $v$, in Section~\ref{s:vel}. Since by now we have shown that $\lambda$ is locally bounded from above and below the Euler-Lagrange equations are uniformly elliptic equations with $W^{1,2}$-coefficients. We prove a priori $L^p$-estimates for such equations in Section~\ref{s:uniformlp}, which might be interesting in their own right -- see Proposition~\ref{pr:vpg2uniformmorrey2}. In Section~\ref{s:vw22} we then obtain successively for $v$ H\"older regularity, Proposition~\ref{pr:initialreg}, $W^{1,p}$-regularity for any $p < \infty$, in Proposition~\ref{pr:vl2peps} and finally $W^{2,2-\eps}$-regularity in Corollary~\ref{co:vW22}. This is the optimal regularity one can hope for without having better estimates on $\lambda$, see \cite{ST13}. So in Section~\ref{s:W22lambda} we turn to improving the regularity $\lambda$, and the already obtained regularity for $v$ allows us to obtain $W^{2,2}$-estimates for $\lambda$ which in turn lead to $W^{2,p}$-estimate for $v$ for any $p < \infty$, see Corollary~\ref{co:w22lw2pv}. Lastly, with the regularity already obtained for $\lambda$ and $v$ we show in Section~\ref{s:lambdac1a} that $\lambda$ solves an elliptic inequality in viscosity sense, and we obtain $C^{1,\alpha}$-regularity of $\lambda$. With this we conclude the promised regularity of $u = \lambda v$.

\section{Boundedness of \texorpdfstring{$\lambda$}{lambda}}\label{s:lbounded}
The scalar function $\lambda$ is a solution to a classical (graph-)obstacle problem, however for the energy
\[
\lambda \mapsto \int |\nabla \lambda|^2 + \int \lambda^2 |\nabla v|^2
\]
But observe that $|\nabla v|^2 \in L^1(\R^n)$, only. In particular, a priori for general $|\nabla v|^2 \in L^1(\R^n)$, we cannot hope that $\lambda$ is very smooth. For now we have to content ourselves with the (local) boundedness of $\lambda$.

\begin{proposition}[Boundedness of $\lambda$]\label{pr:lambdabd}
Let $\lambda$, $v$ be as in Theorem~\ref{th:main2}. Then $\lambda \in L^\infty_{loc}(\B^2)$, that is for any compact set $K \subset \B^2$ we have that $\lambda \in L^\infty(K)$.
\end{proposition}
\begin{proof}
We will show that $\lambda \in L^\infty(B(0,r))$ for any $r \in (0,1)$. Fix such an $r$. By Fubini's theorem, there must be $R \in (r,1)$ such that 
\[
 \|\lambda \|_{H^1(\partial B(0,R))} \aleq \frac{1}{1-r} \|\lambda \|_{H^1(B(0,1))}. 
\]
Since $\partial B(0,r)$ is one-dimensional we have that $H^1(\partial B(0,R))$ embeds in particular into $C^0(\partial B(0,R))$. For simplicity of notation we shall pretend that $R = 1$ and thus assume w.l.o.g.
\[
 K_1 := \|\lambda \|_{L^\infty(\partial B(0,1))} < \infty.
\]
Let $\eta \in C_c^\infty(\B^2,\R_+)$ and let $K > K_1$.
Then for small $\eps > 0$ the following variation of $\lambda$ is admissible
\[
\lambda_\eps := \lambda - \eps \eta (\lambda-K) 
\]
Indeed, by convexity, whenever $\eps \|\eta\|_{\infty} < 1$,
\[
\lambda_\eps = (1-\eps \eta) \lambda + \eps \eta\, K \geq 1 \quad \mbox{a.e. in $\B^2$}.
\]
In particular, the Euler-Lagrange inequality for $\lambda$ implies
\[
\frac{d}{d\eps } \Big|_{\eps  = 0^+} E\left (\lambda_\eps,v \right ) \geq 0,
\]
that is 
\begin{equation}\label{eq:lambdael}
\int_{\B^2} |\nabla \lambda|^2 \eta + \int_{\B^2} (\lambda-K)\nabla \lambda \cdot  \nabla \eta+ \int_{\B^2} \lambda\, \eta\, (\lambda-K)\, |\nabla v|^2 \leq 0.
\end{equation}
We would like to test this inequality with $\eta := (\lambda - K)_+$ Then $\eta \in H^1_0(\B^2)$ -- the zero boundary data stems from the choice of $K \geq K_1$. Moreover,
\[
\nabla \eta = \chi_{\{\lambda \geq K\}}\, \nabla \lambda.
\]
Cf. \cite[Chapter 5, Problem 18, p.308]{Evans}. However $\eta$ may not be bounded, ant the resulting integrals may not converge. So instead for arbitrary $k > K$ we test with
\[
\eta_k := -(\eta - k)_- + k= \min \{\eta-k,0\} + k \in [0,k]
\]
In other words,
\[
\eta_k = \begin{cases}
k \quad &\mbox{in $\{\lambda > K+k\}$}\\
\lambda-K \quad &\mbox{in $\{\lambda < K+k\} \cap \{\lambda > K\}$}\\
0 \quad &\mbox{in $\{\lambda < K\}$}\\
\end{cases}
\]
Now we have $\eta_k \in H^1_0 \cap L^\infty(\B^2,\R_+)$, and consequently, $\eta_k$ is admissible as testfunction in \eqref{eq:lambdael}. Moreover,
\begin{equation}\label{eq:netak}
\nabla \eta_k = \chi_{\{\lambda \in (K,K+k)\}}\, \nabla \lambda.
\end{equation}
We observe that $(K-\lambda) \eta_k \geq 0$ and thus
\[
 \int_{\B^2} \lambda\, \eta\, (K-\lambda)\, |\nabla v|^2 \geq 0.
\]
Moreover, in view of \eqref{eq:netak},
\[
\int_{\B^2} (\lambda-K)\nabla \lambda \cdot  \nabla \eta = \int_{\B^2} (\lambda-K) \chi_{\lambda \in (K,k)} |\nabla \lambda|^2 \geq 0.
\]
Consequently, \eqref{eq:lambdael} implies 
\[
\int_{\B^2} |\nabla \lambda|^2 \eta_k\leq 0,
\]
that is, since $\eta_k \geq 0$,
\[
|\nabla \lambda|^2 \eta_k \equiv 0.
\]
\[
\int_{\B^2} |\nabla \lambda|^2 (\lambda-K)_+ 
\leq 0
\]
This implies
\[
|\nabla \lambda|^2(\lambda-K)_+ \equiv 0,
\]
that is
\[
 |\nabla \brac{(\lambda-K)_+}^2| \equiv 0,
\]
But in view of \eqref{eq:netak} this implies
\[
|\nabla (\eta_k)^2| \equiv 0,
\]
which in turn gives $\eta_k \equiv 0$ (since $\eta_k \in H^1_0(\B^2)$. In particular $\lambda \leq K$ almost everywhere, i.e. $\lambda$ is bounded (recall that $\lambda \geq \lambda_0 \in L^\infty(\B)$ was assumed).
\end{proof}

\section{The Euler-Lagrange equations for \texorpdfstring{$v$}{v}}\label{s:vel}
Now that $\lambda$ is bounded, we start with computing the Euler-Lagrange equations for $v$, which are a weighted version of the spherical harmonic map equation. In particular we obtain a weighted version of Shatah's conservation law \cite{Shatah-1988}, that H\'elein used in \cite{Helein90} to obtain regularity for harmonic maps into spheres. 
\begin{lemma}[Euler-Lagrange equations]\label{la:vel}
Let $\lambda$ and $v$ be as in Theorem~\ref{th:main2}. Then,
\begin{equation}\label{eq:vel}
 \div (\lambda^2 \nabla v^i) = \Omega_{ij}\cdot \lambda^2 \nabla v^j
\end{equation}
with
\[
 \Omega_{ij} = v^j \nabla v^i - v^i \nabla v^j.
\]
Equivalently we also have a weighted version of Shatah's conservation law \cite{Shatah-1988}
\begin{equation}\label{eq:shatah}
 \div (\lambda^2 \Omega_{ij}) = 0.
\end{equation}
\end{lemma}
\begin{proof}
Since $|v| \equiv 1$ we have
\[
 \frac{d}{d\eps} \Big |_{\eps = 0} \frac{v+\eps \psi}{|v+\eps \psi|} = \psi - \langle \psi, v \rangle v,
\]
and consequently, the Euler-Lagrange equations with respect to $v$, can be written as
\[
 \int_{\B^2} \lambda^2\, \nabla v \cdot \nabla \psi = \int_{\B^2} \lambda^2\, \nabla v \cdot \nabla (\langle \psi,v\rangle v)
\]
Now, $v \cdot \nabla v  \equiv 0$ since $|v| \equiv 1$, so
\[
 \nabla v \cdot \nabla (\langle \psi,v\rangle v) = |\nabla v|^2 \langle \psi, v \rangle.
\]
We thus obtain the Euler-Lagrange equation
\[
 \div (\lambda^2 \nabla v^i) = \lambda^2 v^i |\nabla v|^2.
\]
Now we rewrite this equation (using again $v^i \nabla v^i = \frac{1}{2} \nabla |v|^2 \equiv 0$), with he following trick
\[
 v^i |\nabla v|^2 = v^i\, \nabla v^j \cdot \nabla v^j = \brac{v^i\, \nabla v^j \cdot \nabla v^j - v^j\, \nabla v^i} \cdot \nabla v^j
\]
This establishes \eqref{eq:vel}. The conservation law \eqref{eq:shatah} follows now from a direct computation.
\end{proof}

\section{Uniform a priori estimates for critical equations with elliptic \texorpdfstring{$W^{1,2}$}{W12}-coefficients}\label{s:uniformlp}

\begin{proposition}\label{pr:vpg2uniformmorrey2}
Let $2 < p_0 < p_\infty < \infty$ and $\Lambda > 1$. Then there exists a constant $C = C(\Lambda,p_0, p_\infty)$, a small $\eps = \eps(\Lambda,p_0,p_\infty) > 0$ and a small $\alpha = \alpha(\Lambda,p_0,p_\infty) > 0$ so that the following holds.

Let either $p=2$ or $p \in (p_0,p_\infty)$ and $R > 0$. Let 
$v \in W^{1,2}(\B^2,\R^N)$ 
be a solution to
\begin{equation}\label{eq:mrpde}
\div (\lambda^2 \nabla v^i) = \Omega_{ij}\cdot \lambda^2 \nabla v^j \quad \mbox{in $\B^2$}
\end{equation}
where 
$\lambda \in L^\infty \cap W^{1,2}(\B^2)$ satisfies
\begin{equation}\label{eq:lambdape}
\Lambda^{-1} \leq \lambda \leq \Lambda \quad \mbox{almost everywhere in $\B^2$}
\end{equation}
and $\Omega_{ij} \in L^2_{loc}(\B^2,\R^2)$ satisfies
\begin{equation}\label{eq:omegape}
|\Omega| \leq \Lambda |\nabla v| \quad \mbox{almost everywhere in $\B^2$}.
\end{equation}
If $p = 2$ we assume moreover 
\begin{equation}\label{eq:shatah2}
 \div (\lambda^2 \Omega_{ij}) = 0 \quad \mbox{in $\B^2$.}
\end{equation}
Then, if $\nabla v \in L^p_{loc}(\B^2)$ then for any $r < R$ the estimate
\[
\|\nabla v\|_{L^p(B(r))} \leq C\, \brac{\frac{r}{R}}^{\frac{\alpha}{p}} \|\nabla v\|_{L^p(B(R))}
\]
holds for all balls $B(R) \subset \B^2$ on which $v$ and $\lambda$ satisfy
\[
\|\nabla \lambda\|_{L^2(B(R))} + \|\nabla v\|_{L^2(B(R))} \leq \eps.
\]
\end{proposition}

An important ingredient for the $p=2$ case is Wente's Lemma see \cite{Reshetnyak-1968,Wente69,BrC84,Tartar84,Mueller90,CLMS,T97}.
\begin{lemma}[Wente Lemma]\label{la:wente}
Let $B \subset \R^2$ be a ball, and $a, b \in W^{1,2}(B)$. If $w \in W^{1,2}(B(R))$ is a solution to
\[
\begin{cases}
\lap w = \nabla a \cdot \nabla^\perp b \quad &\mbox{in $B$}\\
w = 0 \quad& \mbox{on $\partial B$},
\end{cases}
\]
where $\nabla^\perp = (-\partial_2,\partial_1)^T$, then
\[
\|w\|_{L^\infty(B)} + \|\nabla w\|_{L^2(B)} \leq \|\nabla a\|_{L^2(B)}\, \|\nabla b\|_{L^2(B)}.
\]
\end{lemma}

The proof of Proposition~\ref{pr:vpg2uniformmorrey2} is based on the following estimate.
\begin{lemma}\label{la:vpg2uniformmorrey1}
Let $2 < p_0 < p_\infty < \infty$ and $\Lambda > 1$. Then there exists a constant $C = C(\Lambda,p_0, p_\infty)$ so that the following holds.

Let either $p=2$ or $p \in (p_0,p_\infty)$ and $R > 0$. Let 
$v \in W^{1,2}(B(R),\R^N)$ 
be a solution to \eqref{eq:mrpde} in $B(R)$, where 
$\lambda \in L^\infty \cap W^{1,2}(B(R))$ satisfies \eqref{eq:lambdape} in $B(R)$ and $\Omega_{ij} \in L^2_{loc}(B(R),\R^2)$ satisfies \eqref{eq:omegape}. If $p = 2$ we assume moreover \eqref{eq:shatah2} to hold in $B(R)$. 

Then, if $\nabla v \in L^p(B(R))$ we have the following a priori estimate for any $r \in (0,R]$
\[
\|\nabla v\|_{L^p(B(r))} \leq C\, \brac{\brac{\frac{r}{R}}^{\frac{2}{p}} + \|\nabla \lambda\|_{L^2(B(R))} + \|\nabla v\|_{L^2(B(R))}}\, \|\nabla v\|_{L^p(B(R))}.
\]
\end{lemma}
\begin{proof}
We use Hodge decomposition to obtain
\begin{equation}\label{eq:hodge1}
\lambda^2 \nabla v = \nabla a + \nabla^\perp b + H \quad \mbox{in $B(R)$}
\end{equation}
Here $\nabla^\perp = (-\partial_y,\partial_x)$. Namely, we choose $a, b \in W^{1,2}_0(B(R),\R^N)$, and $H$ harmonic in $B(R)$ so that
\[
\begin{cases}
\lap a = \div(\lambda^2 \nabla v) \quad & \mbox{in $B(R)$}\\
a = 0 \quad &\mbox{on $\partial B(R)$}
\end{cases}, \quad \begin{cases}
\lap b = \curl(\lambda^2 \nabla v) \quad & \mbox{in $B(R)$}\\
b = 0 \quad &\mbox{on $\partial B(R)$}
\end{cases}
\]
From \eqref{eq:mrpde} we find that 
\[
\begin{cases}
\lap a = \Omega_{ij}\cdot \lambda^2 \nabla v^j \quad & \mbox{in $B(R)$}\\
a = 0 \quad &\mbox{on $\partial B(R)$}
\end{cases}
\]
From standard elliptic estimates we then obtain for any $p > 2$
\begin{equation}\label{eq:m1:aest}
\|\nabla a\|_{L^p(B(R))} \aleq \|\lambda\|_{\infty}^2 \|\Omega\|_{L^2(B(R))} \|\nabla v\|_{L^p(B(R))}.
\end{equation}
Of course the constant may depend on $p$ as it blows up for $p \to \infty$ or as $p \to 2$. But it is uniform for $p \in (p_0,p_\infty)$. 
For $p = 2$, we use that by \eqref{eq:shatah2} we have a div-curl structure. Then, Wente's Lemma, Lemma~\ref{la:wente}, implies the same estimate \eqref{eq:m1:aest} for $p=2$.

For $b$ we use compute the curl and find 
\[
\begin{cases}
\lap b = \nabla^\perp(\lambda^2) \nabla v \quad & \mbox{in $B(R)$}\\
b = 0 \quad &\mbox{on $\partial B(R)$}
\end{cases}
\]
Again from standard elliptic estimates for $p>2$ and from Wente's Lemma and the div-curl structure for $p=2$ we obtain the estimate
\[
\|\nabla b\|_{L^p(B(R))} \aleq \|\lambda\|_{L^\infty(B(R))}\, \|\nabla \lambda\|_{L^2(B(R))}\, \|\nabla v\|_{L^p(B(R))}
\]
By the assumptions on $\lambda$ and $\Omega$ we thus get
\[
\|\nabla a\|_{L^p(B(R))} + \|\nabla b\|_{L^p(B(R))} \leq 
C(\Lambda,p_0,p_\infty)\, \left (\|\nabla v\|_{L^2(B(R))} + \|\nabla \lambda\|_{L^2(B(R))}\right )\, \|\nabla v\|_{L^p(B(R))}.
\]
In particular we get from \eqref{eq:hodge1},
\[
\|\nabla v\|_{L^p(B(r))} \leq C(\Lambda)\, \|H\|_{L^p(B(r))} + C(\Lambda,p_0,p_\infty)\, \left (\|\nabla v\|_{L^2(B(R))} + \|\nabla \lambda\|_{L^2(B(R))}\right )\, \|\nabla v\|_{L^p(B(R))},
\]
and
\[
\|H\|_{L^p(B(R))} \leq C(\Lambda) \|\nabla v\|_{L^p(B(R))} + C(\Lambda,p_0,p_\infty)\, \left (\|\nabla v\|_{L^2(B(R))} + \|\nabla \lambda\|_{L^2(B(R))}\right )\, \|\nabla v\|_{L^p(B(R))},
\]
The last ingredient is the harmonicity of $H$, which implies for any $r < R$, see, e.g. \cite[Theorem 2.1, p.78]{G83},
\[
\|H\|_{L^p(B(r))} \aleq \brac{\frac{r}{R}}^{\frac{2}{p}}\, \|H\|_{L^p(B(R))}.
\]
Together, the last three estimates imply the claimed result.
\end{proof}

By choosing $r < \theta^{\frac{p}{n}} R$ for $\theta$ small enough we obtain as a corollary
\begin{corollary}\label{co:vpg2uniformmorrey1b}
Let $2 < p_0 < p_\infty < \infty$ and $\Lambda > 1$. Then there exists a constant $C = C(\Lambda,p_0, p_\infty)$, a small $\eps = \eps(\Lambda,p_0,p_\infty) > 0$ and a small $\theta = \theta(\Lambda,p_0,p_\infty)$ so that the following holds.

Let either $p=2$ or $p \in (p_0,p_\infty)$ and $R > 0$. Let 
$v \in W^{1,2}(B(R),\R^N)$ 
be a solution to  \eqref{eq:mrpde} in $B(R)$, where 
$\lambda \in L^\infty \cap W^{1,2}(B(R))$ satisfies \eqref{eq:lambdape} in $B(R)$ and $\Omega_{ij} \in L^2_{loc}(B(R),\R^2)$ satisfies \eqref{eq:omegape} in $B(R)$.
If $p = 2$ we assume moreover \eqref{eq:shatah2} to hold in $B(R)$.

Then, if $\nabla v \in L^p(B(R))$ and if 
\[
\|\nabla \lambda\|_{L^2(B(R))} + \|\nabla v\|_{L^2(B(R))} < \eps 
\]
then for $\sigma := \theta^{\frac{p}{2}}$ we have
\[
\|\nabla v\|_{L^p(B(\sigma R))} \leq \frac{1}{2} \|\nabla v\|_{L^p(B(R))}
\]
\end{corollary}

\begin{proof}[Proof of Proposition~\ref{pr:vpg2uniformmorrey2}]
The proof now follows from Corollary~\ref{co:vpg2uniformmorrey1b} by iteration.
Pick $r \in (\sigma^{k-1} R,\sigma^{k} R]$ for some $k \in \N$. 

For now let us assume that $k \geq 2$. 
Repeated application of Corollary~\ref{co:vpg2uniformmorrey1b} implies
\[
\|\nabla v\|_{L^p(B(r))} \leq 2^{1-k} \|\nabla v\|_{L^p(B(R))}
\]
Since for our choice of $r$,
\[
2^{1-k} = \sigma^{(k-1) \frac{\log 2}{-\log \sigma}} \leq \brac{\frac{r}{R}}^{\frac{\log 2}{-\log \sigma}}
\]
we have found that
\[
\|\nabla v\|_{L^p(B(r))} \leq \brac{\frac{r}{R}}^{\frac{\log 2}{-\log \sigma}} \|\nabla v\|_{L^p(B(R))}.
\]
Since $\sigma = \theta^{\frac{p}{2}}$ we choose (independently of $p$)
\[
\alpha := 2\frac{\log 2}{- \log \theta}.
\]
That is, we have shown
\[
\|\nabla v\|_{L^p(B(r))} \leq \brac{\frac{r}{R}}^{\frac{\alpha}{p}} \|\nabla v\|_{L^p(B(R))}
\]
holds for any $r \leq \sigma^2 R$. For $r \in (\sigma^2 R,R)$ we use the trivial estimate
\[
\|\nabla v\|_{L^p(B(r))} \leq \|\nabla v\|_{L^p(B(R))} \leq  \sigma^{-2\frac{\alpha}{p}}\, 
\brac{\frac{r}{R}}^{\frac{\alpha}{p}} \|\nabla v\|_{L^p(B(R))}.
\]
Using again that $\sigma = \theta^{\frac{p}{2}}$ we find for any $r \in (\sigma^2 R,R)$
\[
\|\nabla v\|_{L^p(B(r))} \leq \|\nabla v\|_{L^p(B(R))} \leq  \theta^{-2 \alpha}\, 
\brac{\frac{r}{R}}^{\frac{\alpha}{p}} \|\nabla v\|_{L^p(B(R))}
\]
\end{proof}

\section{\texorpdfstring{$W^{2,2-\eps}$}{W22-eps}-regularity of v}\label{s:vw22}
As a consequence of our analysis in the previous section we obtain H\"older continuity of $v$. 
\begin{proposition}[Initial regularity for $v$]\label{pr:initialreg}
Let $v$ and $\lambda$ be as in Theorem~\ref{th:main2}. Then there exists $\alpha > 0$ such that for every compact $K \subset \B^2$ we have
\begin{equation}\label{eq:initmorreyest}
\sup_{B(y_0,r) \subset K} r^{-\frac{\alpha}{2}} \|\nabla v\|_{L^2(B(y_0,r))}< \infty.
\end{equation}
In particular, by Sobolev embedding in $\R^2$, $v \in C^{0,\alpha}_{loc}$.
\end{proposition}
\begin{remark}
The proof of H\"older continuity can be found in the literature: from Lemma~\ref{la:vel}, more precisely \eqref{eq:vel} we obtain that for $\xi^i := \lambda^2 \nabla u^i$ we have
\[
 \div(\xi^i) = \Omega_{ik} \xi^k
\]
H\"older regularity now follows from a distorted version of Rivi\`{e}re's celebrated regularity theorem for systems with antisymmetric potential \cite{R07}. More precisely, \cite[Remark 3.4.]{S10} is applicable -- since $\lambda \in L^\infty_{loc}$ by Proposition~\ref{pr:lambdabd} and using also that by assumption $\inf_{\B^2} \lambda > 0$.

In order to obtain later higher regularity, however, we need the estimate \eqref{eq:initmorreyest}.
\end{remark}
\begin{proof}[Proof of Proposition~\ref{pr:initialreg}]
For $0 < r < R$ let $B(y_0,r) \subset B(y_0,R) \subset \B^2$. Since our result is away from the boundary, by Proposition~\ref{pr:lambdabd} we may assume w.l.o.g. that $\lambda$ is bounded in all of $\B^2$. 

Observe that since $\lambda, v \in W^{1,2}(\B^2)$, by absolute continuity of the integral, for any $\eps > 0$ there exists a radius $R_0 > 0$ such that 
\[
\sup_{B(y_0,\rho) \subset \B^2, \rho < R_0} \|\nabla v\|_{L^2(B(y_0,\rho))} + \|\nabla \lambda\|_{L^2(B(y_0,\rho))} < \eps.
\]
The claim then follows from the a priori estimates of Proposition~\ref{pr:vpg2uniformmorrey2} (for $p=2$) and a covering argument.
\end{proof}

\subsection{Slightly higher integrability of the gradient of v}
The next step is higher integrability of the derivative $\nabla v$,
\begin{proposition}[$W^{1,2+\eps}$-regularity for $v$]\label{pr:vl2peps}
Let $v$ and $\lambda$ be as in Theorem~\ref{th:main2}. Then, there exists $p > 2$ such that $v \in W^{1,p}_{loc}(\B^2)$. 
\end{proposition}
\begin{proof}
We apply Hodge decomposition on a ball $B(R)$. Namely we split
\begin{equation}\label{eq:hodge2}
\lambda^2 \nabla v = \nabla a + \nabla^\perp b + H \quad \mbox{in $B(R)$}
\end{equation}
where $H$ is harmonic in $B(R)$ and $a$ and $b$ are chosen as follows (in view of Lemma~\ref{la:vel}):
\[
\begin{cases}
\lap a = \div(\lambda^2 \nabla v) = \Omega \lambda^2 \nabla v \quad & \mbox{in $B(R)$}\\
a = 0 \quad &\mbox{on $\partial B(R)$}
\end{cases}, \quad \begin{cases}
\lap b = \curl(\lambda^2 \nabla v) = \nabla^\perp \lambda^2 \nabla v \quad & \mbox{in $B(R)$}\\
b = 0 \quad &\mbox{on $\partial B(R)$}
\end{cases}
\]
With the $\alpha$ from Proposition~\ref{pr:initialreg}, the structure of $\Omega$, and boundedness of $\lambda$ we obtain
\[
\sup_{B(r) \subset B(R)} r^{-\alpha} \|\lap a\|_{L^1(B(r))} < \infty.
\]
but for $b$, since $\nabla \lambda \in L^2$ only, we find
\[
\sup_{B(r) \subset B(R)} r^{-\frac{\alpha}{2}} \|\lap b\|_{L^1(B(r))} < \infty.
\]
By (a localized version of) the Sobolev embedding for Morrey spaces, see \cite{A75}, we obtain that for any $p \in [1,\frac{2-\alpha}{1-\alpha})$, $\nabla a$ and $\nabla b$ belong to $L^p_{loc}(B(R))$. Since $\alpha > 0$ we can choose $p > 2$, and since $H$ is harmonic on $B(R)$ and $\lambda$ is bounded away from zero, from \eqref{eq:hodge2} we get $\nabla v \in L^p_{loc}(B(R))$.
\end{proof}

\subsection{On integrability of the gradient of v and \texorpdfstring{$W^{2,2-\eps}$}{W22-eps}-regularity}
Now we can (still only assuming that $\lambda \in W^{1,2}$) bootstrap the regularity for $v$ all the way to $W^{1,p}_{loc}$, $p \in (1,\infty)$. For this we adapt an iteration strategy by Sharp and Topping \cite{ST13}, see also generalizations in \cite{S14,S15}. The main technical ingredient are the uniform a priori estimates in Proposition~\ref{pr:vpg2uniformmorrey2}.

\begin{proposition}[$W^{1,p}$-regularity for $v$ for large $p$]\label{pr:vW1p}
Let $v$ and $\lambda$ be as in Theorem~\ref{th:main2}. Then, for any $p \in (1,\infty)$ we have $v \in W^{1,p}_{loc}(\B^2,\R^N)$.
\end{proposition}

\begin{proof}[Proof of Proposition~\ref{pr:vW1p}]
Fix $p_\infty \in (2,\infty)$. We are going to show that $v \in W^{1,p_\infty}_{loc}(\B^2)$.

By Proposition~\ref{pr:vl2peps} we have $v \in W^{1,p_1}_{loc}(\B^2)$. Set $p_0 := \frac{2+p_1}{2}$, and apply Proposition~\ref{pr:vpg2uniformmorrey2}, then for some (uniform) $\alpha$,
\begin{equation}\label{eq:morreylp2}
\sup_{B(r) \subset K} r^{-\frac{\alpha}{p_1}} \|\nabla v\|_{L^{p_1}(B(r))} < \infty.
\end{equation}
As in  \eqref{eq:hodge2} we apply Hodge decomposition on some ball $B(R) \subset \B^2$.
\begin{equation}\label{eq:hodge3}
\lambda^2 \nabla v = \nabla a + \nabla^\perp b + H \quad \mbox{in $B(R)$}
\end{equation}
where $H$ is harmonic in $B(R)$ and in view of Lemma~\ref{la:vel} we have
\[
\begin{cases}
\lap a = \Omega \lambda^2 \nabla v \quad & \mbox{in $B(R)$}\\
a = 0 \quad &\mbox{on $\partial B(R)$}
\end{cases}, \quad \begin{cases}
\lap b = \nabla^\perp \lambda^2 \nabla v \quad & \mbox{in $B(R)$}\\
b = 0 \quad &\mbox{on $\partial B(R)$}
\end{cases}
\]
From \eqref{eq:morreylp2} we obtain
\[
\sup_{B(r) \subset B(R)} r^{-2\frac{\alpha}{p_1}} \|\lap a\|_{L^{\frac{p_1}{2}}(B(r))} < \infty
\]
and (recall that we only have $\nabla \lambda \in L^2$),
\[
\sup_{B(r) \subset B(R)} r^{-\frac{\alpha}{p_1}} \|\lap b\|_{L^{\frac{2p_1}{p_1+2}}(B(r))} < \infty
\]
Observe that since $p_1 > 2$ we have $\frac{p_1}{2} > q_1 := \frac{2p_1}{2p_1+2} > 1$. That is we have,
\[
\sup_{B(r) \subset B(R)} r^{-\frac{\alpha}{p_1}} \brac{\|\lap a\|_{L^{q_1}(B(r))} + \|\lap b\|_{L^{q_1}(B(r))}} < \infty
\]
Again we use the Sobolev embedding on Morrey spaces, see \cite{A75}. For
\[
\frac{1}{p_2} := \frac{1}{q_1} - \frac{1}{2-\frac{\alpha}{p_1} q_1} = \frac{1}{p_1} -\frac{\alpha}{4p_1+4-2\alpha} 
\]
we get
\[
\sup_{B(r) \subset B(R/2)} r^{- \frac{\frac{\alpha}{p_1+1}}{p_2}} \brac{\|\nabla a\|_{L^{p_2}(B(r))} + \|\nabla b\|_{L^{p_2}(B(r))}} < \infty.
\]
In particular, from \eqref{eq:hodge3} and harmonicity of $H$ we get $v \in L^{p_2}(B(R/2))$. By a covering argument we conclude that $v \in W^{1,p_2}_{loc}(\B^2)$.

So we define a sequence $(p_i)_{i}$ by 
\[
\frac{1}{p_{i+1}} :=  \frac{1}{p_i} -\frac{\alpha}{4p_i+4-2\alpha}.
\]
By induction we obtain from Proposition~\ref{pr:vpg2uniformmorrey2} $v \in W^{1,p_{i+1},\R^N}_{loc}(\B^2)$ for every $i \in \N$ such that $p_i < p_\infty$. The important point is that $\alpha$ is uniform and does not depend on each $i$.

Clearly $p_{i+1} \geq p_{i}$ and $\lim_{i \to \infty} p_i = \infty$. That is there exists $i_0 \in \N$ such that $p_i < p_\infty$ and $p_{i+1} > p_\infty$. That means that $v \in W^{1,p_\infty}_{loc}(\B^2)$.
\end{proof}
As a direct corollary from Proposition~\ref{pr:vW1p} and the Euler-Lagrange equations in Lemma~\ref{la:vel} we obtain $W^{2,2-\eps}_{loc}$-regularity for $v$. Observe that in view of the counterexamples in \cite{ST13} this is the best regularity for $v$ one can hope for without having further improvements on the regularity of $\lambda$.
\begin{corollary}\label{co:vW22}
Let $v$ and $\lambda$ be as in Theorem~\ref{th:main2}. Then, for any $q \in (1,2)$ we have $v \in W^{2,q}_{loc}$.
\end{corollary}
\begin{proof}
From Proposition~\ref{pr:vW1p} we have that $v \in W^{1,p}_{loc}(\B^2)$ for any $p \in (1,\infty)$. From Lemma~\ref{la:vel} we thus get that for any $q \in (1,2)$ -- recall that $\nabla \lambda \in L^2(\B^2)$ --
\[
\div(\lambda^2 \nabla v) \in L^q_{loc}(\B^2)
\]
Now
\[
\lap v = \div(\lambda^{-2}  \lambda^2 \nabla v) = \nabla \lambda^{-2}\, \lambda^2\, \nabla v + L^q_{loc}(\B^2).
\]
Since $\inf_{\B^2} \lambda > 0$ we have that $\lambda^{-2} \in H^1(\B^2)$ and thus
\[
\lap v \in L^q_{loc}(\B^2).
\]
Standard elliptic estimates imply now $v \in W^{2,q}_{loc}(\B^2,\R^N)$.
\end{proof}
\section{\texorpdfstring{$W^{2,2}$}{W22}-regularity for \texorpdfstring{$\lambda$}{lambda}}\label{s:W22lambda}
By now, for $\lambda$ and $v$ as in Theorem~\ref{th:main2} we have shown in Lemma~\ref{pr:lambdabd} that $\lambda \in L^\infty_{loc}(\B^2)$ and in Corollary~\ref{co:vW22} that $v \in W^{2,q}_{loc}(\B^2)$ for any $q \in (1,2)$. 

Recall that we assume that $\lambda \geq 1$. It will be notationally convenient to work with $\mu := \lambda - 1$, which is a critical point of the energy
\[
\int |\nabla \mu|^2 + \int (\mu^2 - 2 \mu)|\nabla v|^2 \quad 
\]
So in the following we are going to consider the regularity of critical points $\mu \in L^\infty(\B^2,[0,\infty))$
\[
F(\mu) := \int |\nabla \mu|^2 + \int (\mu^2 - 2 \mu)\, g \quad \mbox{subject to $\mu \geq 0$}
\]
where $g \in W^{1,q}_{loc}(\B^2)$ for any $q < 2$, in particular $g \in L^2_{loc}(\B^2)$.

First we observe the variational inequality.
\begin{lemma}
Let $\mu$ as above, i.e. a critical point of $F$. Then, for any $\varphi \in C_c^\infty(\B^2)$ such that $\varphi \geq 0$ we have
\begin{equation}\label{eq:muel}
\int_{\B^2} \nabla \mu \cdot \nabla (\varphi-\mu) + \int_{\B^2} \mu (\varphi-\mu) g - \int_{\B^2} (\varphi-\mu) g\geq 0.
\end{equation}
\end{lemma}
\begin{proof}
This follows using the variation
\[
\mu_\eps := \mu + \eps (\varphi-\mu).
\]
\end{proof}
The variational inequality \eqref{eq:muel} for $\mu$ is almost of the form of variational inequalities considered e.g. in \cite[(2.6)]{F71}, where Frehse showed how Nirenberg's method of discretely differentiating partial differential equations can be adapted to variational inequalities. Indeed, the only additional term that does not appear in \cite[(2.6)]{F71} is $\int_{\B^2} \mu (\varphi-\mu) g$. So we (slightly) adapt Frehse's argument to obtain
\begin{proposition}
Let $\mu \in W^{1,2}(\B^2) \cap L^\infty_{loc}(\B^2)$ as above, i.e. a critical point of $F$. If $g \in W^{1,q}_{loc}(\B^2)$ for any $q < 2$, then $\mu \in W^{2,2}_{loc}(\B^2)$.
\end{proposition}\label{pr:vw22}
We now follow closely Frehse's argument in \cite{F71}, and only prove the differences.

Firstly, we introduce first and second order differential quotients,
\[
\delta_{i;h} \mu(x) := \frac{\mu(x+he_i) - \mu(x)}{h},
\]
\[
\delta_{i,j;h} \mu(x) := \frac{\mu(x+he_i) + \mu(x-he_j) - 2\mu(x)}{h^2}
\]
The main first observation in \cite[Hilfssatz 1]{F71}
\begin{lemma}
Let $\eta \in C_c^\infty(\B^2)$, $\eta \geq 0$. Then for any $h < \dist(\supp \eta,\partial \B^2)$ we have for $i = 1,2$,
\[
\int_{\B^2} \nabla \mu \cdot \nabla (\eta^2 \delta_{i,i;h}\mu) + \int_{\B^2} \mu (\eta^2 \delta_{i,i;h}\mu) g - \int_{\B^2} (\eta^2 \delta_{i,i;h}\mu) g\geq 0.
\]
\end{lemma}
\begin{proof}
Observe that for $\eps \ll h$ we have that
\[
\mu_\eps := \mu + \eps \eta^2 \delta_{i,j;h}\mu \geq 0.
\]
In particular, $\mu_\eps$ is a permissible variation of $\mu$, and the claim follows.
\end{proof}

The only term that we have to estimate additionally to Frehse's \cite{F71} is the following:
\begin{lemma}\label{la:muet}
\[
\int_{\B^2} \mu \eta^2\, \delta_{i,i;h}\mu\, g \aleq C\, \brac{1+\|\nabla \delta_{i,h}(\eta \mu )\|_{L^2(\B^2)}}
\]
where $C$ depends on $\supp \eta$, $\|\mu\|_{L^\infty}$, $\|\nabla \mu\|_{L^2}$, $\|\nabla \eta\|_{L^\infty}$, $\|g\|_{L^p}$, for $p$ sufficiently close to $\infty$ and $\|\nabla g\|_{W^{1,q}}$ for a $q < 2$ sufficiently close to $2$.
\end{lemma}
\begin{proof}
First, a standard application of the discrete Leibniz rule,
\[
\int_{\B^2} \mu \eta^2\, \delta_{i,i;h}\mu\, g = -\int_{\B^2} \mu\, \delta_{i,-h} \delta_{i,h}(\eta\, \mu)\, g + C.
\]
Thus, with the discrete integration by parts rule we obtain for any $q \in (1,\infty)$,
\[
\|\nabla (\eta \mu  g)\|_{L^{q}}\, \|\delta_{i,h} (\eta \mu) \|_{L^{\frac{q}{q-1}}}
\]
For $q < 2$ we have
\[
\|\nabla (\eta \mu  g)\|_{L^{q}} \aleq C.
\]
On the other hand, since $\delta_{i,h} (\eta \mu)$ has compact support, by Sobolev-Poincar\`e-embedding (since we are in two dimensions) for any $q \in (1,2)$,
\[
\|\delta_{i,h} (\eta \mu) \|_{L^{\frac{q}{q-1}}} \aleq \|\nabla \delta_{i,h} (\eta \mu) \|_{L^{2}}.
\]
\end{proof}

\begin{proof}[Proof of Proposition~\ref{pr:vw22}]
Following word-by-word the Frehse's argument in \cite{F71}, using additionally the estimate Lemma~\ref{la:muet} we obtain, cf. \cite[p. 149]{F71},
\[
\|\nabla (\delta_{i,h} (\mu \eta))\|_{L^2}^2 \leq C\, \brac{1+\|\nabla (\delta_{i,h} (\mu \eta))\|_{L^2}}.
\]
From Young's inequality we obtain
\[
\|\nabla (\delta_{i,h} (\mu \eta))\|_{L^2}^2 \leq C +4C^2+ \frac{1}{2} \|\nabla (\delta_{i,h} (\mu \eta))\|_{L^2}^2 
\]
and thus
we obtain a bound on $\|\nabla (\delta_{i,h} (\mu \eta))\|_{L^2}^2$ independent of $h$. Letting $h \to 0$ we get that 
\[
\nabla \partial_i (\mu \eta) \in L^2,
\]
which readily leads to $u \in W^{2,2}$ in the set where $\eta \equiv 1$. Taking $\eta \in C_c^\infty(\B^2)$ and $\eta \equiv 1$ on $K \subset \B^2$, $K$ compact, we get that $\mu \in W^{2,2}_{loc}(K)$. This holds for any compact set $K \subset \B^2$, so the claim is proven.
\end{proof}

\begin{corollary}\label{co:w22lw2pv}
Let $v$ and $\lambda$ be as in Theorem~\ref{th:main2}. Then $\lambda \in W^{2,2}_{loc}(\B^2)$, $v \in W_{loc}^{2,p}(\B^2,\R^N)$ for any $p < \infty$. In particular $\lambda \in C^{0,\alpha}$ and $v \in C^{1,\alpha}$ for any $\alpha < 1$.
\end{corollary}
\begin{proof}
By Corollary~\ref{co:vW22} we have that $|\nabla v|^2 \in W^{1,q}_{loc}$ for any $q < 2$. Thus Proposition~\ref{pr:vw22} is applicable to $\mu = \lambda-1$, and we get that $\lambda = \mu+1\in W^{2,2}_{loc}(\B^2)$.

To obtain $W^{2,p}_{loc}$-regularity for $v$, we consider again the equations for $v$, \eqref{eq:vel}, and note that 
\[
\div(\lambda^2 \nabla v) \in L^p_{loc}(\B^2)
\]
Moreover, since $\lambda \geq 1$, we compute
\[
\lap v = \div(\lambda^{-2} \lambda^2 \nabla v) =  \lambda^{-} \nabla \lambda\, \nabla v + \lambda^{-2} \div(\lambda^2 \nabla v).
\]
Since $\nabla \lambda \in W^{1,2}_{loc}$ and $\nabla v \in W^{1,q}_{loc}$ for any $q < 2$, we obtain that $\lap v \in L^p_{loc}$ and consequently standard Calderon-Zygmund theory implies that $\nabla^2 v \in L^p_{loc}$.
\end{proof}

\section{On \texorpdfstring{$C^{1,\alpha}$}{C1alpha}-regularity for \texorpdfstring{$\lambda$}{lambda}}\label{s:lambdac1a}
At this stage we have that $v \in W^{2,p}_{loc}(\B^2,\R^N)$ for all $p \in (1,\infty)$ and $\lambda \in W^{2,2}_{loc}(\B^2)$. Observe that this implies in particular that $\lambda$ is continuous. Since the obstacle condition $\lambda \geq 1$ is pointwise, the theory of viscosity solutions (see e.g. \cite{CC95,K12}) is more suitable now.

\begin{proposition}\label{pr:lvs}
There exists a constant $\Lambda > 0$ such that $\lambda$ solve in viscosity sense the inequalities
\[
0 \leq \lap \lambda \leq \Lambda \quad \mbox{in $\B^2$}.
\]
\end{proposition}
This Proposition is a consequence of Lemma~\ref{la:visc:lambdag} and Lemma~\ref{la:visc:lambdal} below. The first observation is that $\lambda$ is smooth in the open set $\{\lambda > 1\}$.
\begin{lemma}\label{la:lvsmooth}
We have $\lambda, v \in C^\infty(\{\lambda > 1\})$ and we have
\begin{equation}\label{eq:leq}
\lap \lambda = \lambda |\nabla v|^2 \quad \mbox{pointwise in  $\{\lambda > 1\}$}
\end{equation}
In view of Corollary~\ref{co:w22lw2pv} there exists in particular $\Lambda > 0$ such that 
\[
\lap \lambda \leq \Lambda \quad \mbox{in  $\{\lambda > 1\}$}.
\]
\end{lemma}
\begin{proof}
We revert our attention to $u := \lambda v$ and show that $u \in C^\infty(\{\lambda > 1\})$. Let $x_0 \in \{\lambda > 1\}$. Then, since $\lambda$ is continuous, there exists a ball  $B := B(x_0)$ such that $\overline{B} \subset \{\lambda > 1\}$. But this implies that for any $\varphi \in C_c^\infty(B)$ for all suitably small
\[
u_\eps := u + \eps \varphi
\]
is a permissible variation of the Dirichlet energy, setting $\lambda_\eps := |u_\eps| > 0$ and $v_\eps := \frac{u_\eps}{|u_\eps|^2}$. That is,
\[
\frac{d}{d\eps} \Big |_{\eps = 0} \int |\nabla u_\eps|^2 = 0.
\]
But this implies
\[
\lap u = 0 \quad \mbox{in $\{\lambda > 1\}$},
\]
so in particular $\lambda = |u| \in C^\infty\left (\{\lambda > 1\}\right )$ and $v = \frac{u}{|u|} \in C^\infty\left (\{\lambda > 1\} \right )$.

The equation \eqref{eq:leq} follows now from the variation $\lambda_\eps := \lambda + \eps \psi$ for arbitrary $\psi \in C_c^\infty(\{\lambda > 1\})$.
\end{proof}

\begin{lemma}\label{la:visc:lambdag}
We have in viscosity sense
\[
\lap \lambda \geq 0 \quad \mbox{in $\B^2$}.
\]
\end{lemma}
\begin{proof}
Let $x_0 \in \B^2$. If $\lambda(x_0) > 1$, then the claim follows immediately from Lemma~\ref{la:lvsmooth}, since smooth solutions are viscosity solutions.

So assume that $\lambda(x_0) = 1$. For any test-function $\varphi \geq \lambda$ such that $\varphi(x_0) = 1$ we have in particular
\[
\varphi \geq 1, \quad \mbox{and} \quad \varphi(x_0) = 1.
\]
That is, $\varphi$ attains its minimum at $x_0$ and thus $\lap \varphi \geq 0$.
\end{proof}

\begin{lemma}\label{la:visc:lambdal}
For $\lambda$ as above we have in viscosity sense.
\[
\lap \lambda \leq \Lambda \quad \mbox{in $\B^2$}.
\]
\end{lemma}
\begin{proof}
By the variation $\lambda_\eps := \lambda + \eps \varphi$ for $\varphi \in C_c^\infty(\B^2)$ and $\varphi \geq 0$ we get the variational inequality
\[
\int \nabla \lambda \cdot \nabla \varphi + \lambda \varphi |\nabla v|^2 \geq 0.
\]
Let $\eta \in C_c^\infty(B(0,1))$ be the usual bump function, $\eta \equiv 1$ in $B(0,\frac{1}{2})$, $\eta(-x) = \eta(x)$, $\eta \geq 0$ and $\int \eta = 1$. We set $\eta_\eps := \eps^{-2} \eta(\cdot/\eps)$. We denote $\lambda_\eps := \eta_\eps \ast \lambda$ and have for any fixed testfunction $\varphi \geq 0$ (if $\eps$ is small enough then $\varphi \ast \eta_\eps \geq 0$ is permissible as a test function)
\[
\begin{split}
\int \nabla \lambda_\eps \cdot \nabla \varphi &=  \int \nabla \lambda \cdot \nabla (\varphi \ast \eta_\eps)\\
\geq& - \int \lambda (\varphi\ast \eta_\eps) |\nabla v|^2 \\
=& - \int \varphi \brac{\lambda |\nabla v|^2}\ast \eta_\eps  \\
\end{split}
\]
Since, in view of Corollary~\ref{co:w22lw2pv} we have
\[
\|\brac{\lambda |\nabla v|^2}\ast \eta_\eps \|_{L^\infty} \leq 
\|\lambda |\nabla v|^2\|_{L^\infty} =: \Lambda < \infty
\]
we find that 
\begin{equation}\label{eq:lepde}
\lap \lambda_\eps \leq \Lambda \quad \mbox{in $\B_{1-2\eps}$}
\end{equation}
This inequality holds in pointwise and viscosity sense, since $\lambda_\eps$ is smooth.

On the other hand, since $\lambda$ is H\"older continuous, we have that $\lambda_\eps$ converges locally uniformly to $\lambda$ as $\eps \to 0$. This implies, e.g. as in \cite[Lemma 2.4]{S18}, that also $\lambda$ satisfies \eqref{eq:lepde} in viscosity sense.

\end{proof}

As a consequence of the regularity theory of viscosity solutions to elliptic partial differential inequalities, see e.g. \cite{S18}, we obtain
\begin{corollary}
Let be $\lambda$ as above, then $\lambda \in C^{1,\alpha}$ for any $\alpha < 1$.
\end{corollary}

\section{Adaptations for the proof of Theorem~\ref{th:mainSO}}\label{s:CO}
For matrices $A,B \in \R^{N \times N}$ we denote the Hilbert-Schmidt scalar product by
\[
A:B := \sum_{i,j = 1}^\infty A_{ij} B_{ij}.
\]
Now as in the sphere case, where we have $u \cdot \nabla u = 0$ if $|u| =1$ almost everywhere, if $P \in SO(N)$ almost everywhere then
\[
\nabla P : P = P^T \nabla P: I_{N \times N} = 0,
\]
since $P^T \nabla P$ is antisymmetric and the identity matrix $I_{N \times N}$ is symmetric.

In particular we have for $\lambda \in H^1(\B^2)$ and $P \in H^1(\B^2,SO(N))$,
\[
|\nabla (\lambda P)|^2 = |\nabla \lambda|^2 + \lambda^2 |\nabla P|^2.
\]
We conclude that we have to consider critical points of the energy
\[
E(\lambda,P) = |\nabla \lambda|^2 + \lambda^2 |\nabla P|^2.
\]
So we see that we get the analogue of Lemma~\ref{la:vel}. Now regularity estimates are almost verbatim of what we have here.
\begin{lemma}[Euler-Lagrange equations]\label{la:Pel}
Let $\lambda$ and $P$ be as in Theorem~\ref{th:mainSO}. Then,
\[
 \div (\lambda^2 \nabla P) = \lambda^2 \Omega \nabla P. 
\]
with
\[
 \Omega = -P^T \nabla P.
\]
Equivalently we also have the conservation law
\[
\div(\lambda^2 P^T \nabla P) = 0.
\]
\end{lemma}
\begin{proof}
A permissible variation for $P$ is $P_\eps := P e^{\eps \alpha \varphi}$ where $\alpha \in so(N)$ is antisymmetric and $\varphi \in C_c^\infty(\B^2)$. This leads to
\[
\int \lambda^2 \nabla P: \nabla (P\alpha \varphi) = 0.
\]
Observe that for antisymmetric $\alpha$ we readily have
\[
\nabla P: \nabla P \alpha = 0
\]
Thus, the Euler-Lagrange equations for variations in $P$ are 
\[
\div(\lambda^2 P^T \nabla P):\alpha = 0.
\]
This holds for any antisymmetric matrix $\alpha \in so(N)$. Using that $P^T \nabla P$ is also antisymmetric, we thus get
\[
\div(\lambda^2 P^T \nabla P) = 0.
\]
We can equivalently rewrite this as
\[
\div(\lambda^2 \nabla P) = \div(\lambda^2 P P^T \nabla P) =\lambda^2 \nabla P P^T \nabla P. 
\]
Using that $\nabla P P^T = \nabla (P P^T)-P^T \nabla P=-P^T \nabla P$ we get the claim.
\end{proof}

\subsection*{Acknowledgment}
A.S. was supported by the German Research Foundation (DFG) through grant no.~SCHI-1257-3-1, by the Daimler and Benz foundation through grant no. 32-11/16, as well as the Simons foundation through grant no 579261. Part of this work was carried out while A.S. was visiting Chulalongkorn University whose hospitality is gratefully acknowledged.

\bibliographystyle{abbrv}%
\bibliography{bib}%

\begin{thebibliography}{10}

\bibitem{A75}
D.~R. Adams.
\newblock A note on {R}iesz potentials.
\newblock {\em Duke Math. J.}, 42(4):765--778, 1975.

\bibitem{BrC84}
H.~Brezis and J.-M. Coron.
\newblock Multiple solutions of {$H$}-systems and {R}ellich's conjecture.
\newblock {\em Comm. Pure Appl. Math.}, 37(2):\ 149--187, 1984.

\bibitem{CC95}
L.~Caffarelli and X.~Cabr\'e.
\newblock {\em Fully nonlinear elliptic equations}, volume~43 of {\em AMS
  Colloquium Publications}.
\newblock AMS, Providence, RI, 1995.

\bibitem{CLMS}
R.~Coifman, P.-L. Lions, Y.~Meyer, and S.~Semmes.
\newblock Compensated compactness and {H}ardy spaces.
\newblock {\em J.~Math.~Pures Appl.,~IX.~S\'er.}, 72(3):\ 247--286, 1993.

\bibitem{DF86}
F.~Duzaar and M.~Fuchs.
\newblock Variational problems with nonconvex obstacles and an
  integral-constraint for vector-valued functions.
\newblock {\em Math. Z.}, 191(4):585--591, 1986.

\bibitem{Evans}
L.~C. Evans.
\newblock {\em Partial differential equations}, volume~19 of {\em Graduate
  Studies in Mathematics}.
\newblock American Mathematical Society, Providence, RI, second edition, 2010.

\bibitem{F71}
J.~Frehse.
\newblock Zum {D}ifferenzierbarkeitsproblem bei {V}ariationsungleichungen
  h\"{o}herer {O}rdnung.
\newblock {\em Abh. Math. Sem. Univ. Hamburg}, 36:140--149, 1971.
\newblock Collection of articles dedicated to Lothar Collatz on his sixtieth
  birthday.

\bibitem{G83}
M.~Giaquinta.
\newblock {\em Multiple integrals in the calculus of variations and nonlinear
  elliptic systems}, volume 105 of {\em Annals of Mathematics Studies}.
\newblock Princeton University Press, Princeton, NJ, 1983.

\bibitem{Helein90}
F.~H{\'{e}}lein.
\newblock {R\'{e}gularit\'{e} des applications faiblement harmoniques entre une
  surface et une sph\`{e}re}.
\newblock {\em C.R. Acad. Sci. Paris 311, S\'{e}rie I}, pages 519--524, 1990.

\bibitem{H73}
S.~Hildebrandt.
\newblock Interior {$C^{1+\alpha }$}-regularity of solutions of two-dimensional
  variational problems with obstacles.
\newblock {\em Math. Z.}, 131:233--240, 1973.

\bibitem{HN82}
S.~Hildebrandt and J.~C.~C. Nitsche.
\newblock A uniqueness theorem for surfaces of least area with partially free
  boundaries on obstacles.
\newblock {\em Arch. Rational Mech. Anal.}, 79(3):189--218, 1982.

\bibitem{K12}
S.~Koike.
\newblock {\em A beginner's guide to the theory of viscosity solutions}.
\newblock www.math.tohoku.ac.jp/$\sim$koike/evis2012version.pdf. 2012.

\bibitem{Mueller90}
S.~M{\"u}ller.
\newblock Higher integrability of determinants and weak convergence in {$L\sp
  1$}.
\newblock {\em J.~Reine Angew.~Math.}, 412:\ 20--34, 1990.

\bibitem{Reshetnyak-1968}
Y.~G. Reshetnyak.
\newblock Stability theorems for mappings with bounded excersions.
\newblock {\em Siberian Mathematical Journal}, 9(3):499--512, 1968.

\bibitem{R95}
T.~Rivi\`ere.
\newblock Everywhere discontinuous harmonic maps into spheres.
\newblock {\em Acta Math.}, 175(2):197--226, 1995.

\bibitem{R07}
T.~Rivi\`ere.
\newblock Conservation laws for conformally invariant variational problems.
\newblock {\em Invent. Math.}, 168(1):1--22, 2007.

\bibitem{S10}
A.~Schikorra.
\newblock A remark on gauge transformations and the moving frame method.
\newblock {\em Ann. Inst. H. Poincar\'{e} Anal. Non Lin\'{e}aire},
  27(2):503--515, 2010.

\bibitem{S15}
A.~Schikorra.
\newblock {$\varepsilon$}-regularity for systems involving non-local,
  antisymmetric operators.
\newblock {\em Calc. Var. Partial Differential Equations}, 54(4):3531--3570,
  2015.

\bibitem{S18}
A.~Schikorra.
\newblock {A remark on $C^{1,\alpha}$-regularity for differential inequalities
  in viscosity sense}.
\newblock {\em preprint}, 2018.

\bibitem{SU84}
R.~Schoen and K.~Uhlenbeck.
\newblock Regularity of minimizing harmonic maps into the sphere.
\newblock {\em Invent. Math.}, 78(1):89--100, 1984.

\bibitem{S14}
B.~Sharp.
\newblock Higher integrability for solutions to a system of critical elliptic
  {PDE}.
\newblock {\em Methods Appl. Anal.}, 21(2):221--240, 2014.

\bibitem{ST13}
B.~Sharp and P.~Topping.
\newblock Decay estimates for {R}ivi\`ere's equation, with applications to
  regularity and compactness.
\newblock {\em Trans. Amer. Math. Soc.}, 365(5):2317--2339, 2013.

\bibitem{Shatah-1988}
J.~Shatah.
\newblock Weak solutions and development of singularities of the {${\rm
  SU}(2)$} {$\sigma$}-model.
\newblock {\em Comm. Pure Appl. Math.}, 41(4):459--469, 1988.

\bibitem{Tartar84}
L.~Tartar.
\newblock {Remarks on Oscillations and Stokes' Equation}.
\newblock {\em Lecture Notes in Physics, 230, macroscopic Modelling of
  Turbulent Flows, Proceedings, Sophia-Antipolis, France}, pages 24--31, 1984.

\bibitem{T97}
P.~Topping.
\newblock The optimal constant in {W}ente's {$L^\infty$} estimate.
\newblock {\em Comment. Math. Helv.}, 72(2):316--328, 1997.

\bibitem{Wente69}
H.~C. Wente.
\newblock An existence theorem for surfaces of constant mean curvature.
\newblock {\em J.~Math.~Anal.~Appl.}, 26:\ 318--344, 1969.

\end{thebibliography}

\end{document}